\documentclass[11pt]{amsart}
\usepackage{a4wide,amsmath,amsthm,amsfonts,amssymb,color}%,mathrsfs,showlabels}

\usepackage[all]{xy}

\newcommand*{\vv}[1]{\vec{\mkern1mu#1}\!}

\newcommand{\Fin}{\mathcal{F}\kern-1pt\mathit{in}}
\newcommand{\w}{\omega}
\newcommand{\IN}{\mathbb N}

\newcommand{\F}{\mathcal F}

\newcommand{\U}{\mathcal U}
\newcommand{\V}{\mathcal V}

\newcommand{\Zar}{\mathcal W^\bullet}

\newcommand{\pr}{\mathrm{pr}}
\newcommand{\Ra}{\Rightarrow}

\newcommand{\semilat}{\Diamond}

\newcommand{\setmap}{\multimap}

\newtheorem{theorem}{Theorem}[section]
\newtheorem{proposition}[theorem]{Proposition}
\newtheorem{theorem*}{Theorem}
\newtheorem{lemma}[theorem]{Lemma}
\newtheorem{corollary}[theorem]{Corollary}
\newtheorem{claim}[theorem]{Claim}
\newtheorem{example}[theorem]{Example}

\theoremstyle{definition}
\newtheorem{definition}[theorem]{Definition}

\newtheorem{problem}[theorem]{Problem}

\title{Complete topologized posets and semilattices}
\author{Taras Banakh, Serhii Bardyla}
\address{T.Banakh: Ivan Franko National University of Lviv (Ukraine) and Jan Kochanowski University in Kielce (Poland)}
\email{t.o.banakh@gmail.com}
\address{S.~Bardyla: Institute of Mathematics, Kurt G\"{o}del Research Center, Vienna (Austria)}
\thanks{The second author was supported by the Austrian Science Fund FWF (Grant  I 3709-N35).}
\email{sbardyla@yahoo.com}

\begin{document}

\begin{abstract}In this paper we discuss the notion of completeness of a topologized poset and survey some recent results on closedness properties of complete topologized semilattices.
 \end{abstract}
\maketitle

\section{Introduction}

In this paper we discuss a notion of completeness for topologized posets and semilattices.

By a {\em poset} we understand a set $X$ endowed with a partial order $\le$.
A {\em topologized poset} is a poset endowed with a topology.

A topologized poset $X$ is defined to be {\em complete} if each nonempty chain $C$ in $X$ has $\inf C\in\bar C$ and $\sup C\in\bar C$, where $\bar C$ stands for the closure of $C$ in $X$. More details on this definition can be found in Section~\ref{s:cpo}, where we prove that complete topologized posets can be equivalently defined using directed sets instead of chains. In Section~\ref{s:oper} we obtain some results on the preservation of competeness by operations over topologized posets.

In Section~\ref{s:cc} we study the interplay between complete and chain-compact topologized posets and in Section~\ref{s:5} we study complete topologized semilattices. In Section~\ref{s:ac} we survey some known results on the absolute closedness of complete semitopological semilattices and in Section~\ref{s:o} we survey recent results on the closedness of the partial order in (complete)  semitopological semilattices.

\section{The completeness of topologized posets}\label{s:cpo}

In this section we define the notion of a complete topologized poset, which is a topological counterpart of the standard notion of a complete poset, see \cite{Bible}.
First we recall some concepts and notations from the theory of partially ordered sets.

A subset $C$ of a poset $(X,\le)$ is called a {\em chain} if any two points $x,y\in X$ are comparable in the partial order of $X$. This can be written as $y\in{\updownarrow}x$ where
$${\uparrow}x:=\{y\in C:x\le y\},\;\;{\downarrow}x:=\{y\in C:y\le x\}, \;\;\mbox{and}\;\;{\updownarrow}x:=({\uparrow}x)\cup({\downarrow}x).$$
A poset $X$ is called {\em linear} if $X$ is a chain in $X$.

A subset $D$ of a poset $(X,\le)$ is called {\em up-directed} (resp. {\em down-directed\/}) if for any elements $x,y\in D$ there exists an element $z\in D$ such that $x\le z$ and $y\le z$ (resp. $z\le x$ and $z\le y$). It is clear that each chain is both up-directed and down-directed.

A poset $X$ is defined to be
\begin{itemize}
\item {\em up-complete} if any nonempty up-directed set $D\subseteq X$ has $\sup D$ in $X$;
\item {\em down-complete} if any nonempty down-directed set $D\subseteq X$ has $\inf D$ in $X$;
\item {\em complete} if $X$ is up-complete and down-complete;
\item {\em a complete lattice} if any nonempty subset $A\subseteq X$ has $\sup A$ and $\inf A$ in $X$.
\end{itemize}

%In Corollary~\ref{??} we shall prove that a poset $X$ is up-complete (resp. down-complete) if and only if each nonempty chain $C\subseteq X$ has $\sup C$ (resp. $\inf C$) in $X$. This implies that a poset $X$ is complete if and only if each nonempty chain $C\subseteq X$ has $\inf C$ and $\sup C$ in $X$.

In the following definition we introduce topological counterparts of these notions.

\begin{definition} Let $\kappa$ be a cardinal. A topologized poset $X$ is defined to be
 \begin{itemize}
\item {\em up-complete} (resp. {\em ${\uparrow}\kappa$-complete})  if each nonempty up-directed set $D\subseteq X$ (of cardinality $|D|\le\kappa$) has $\sup D\in\bar D$ in $X$;
\item {\em down-complete} (resp. {\em ${\downarrow}\kappa$-complete})  if each nonempty down-directed set $D\subseteq X$ (of cardinality $|D|\le\kappa$) has $\inf D\in\bar D$ in $X$;
\item {\em complete} (resp. {\em ${\updownarrow}\kappa$-complete}) if $X$ is up-complete and down-complete (resp. ${\uparrow}\kappa$-complete and ${\downarrow}\kappa$-complete).
\end{itemize}
\end{definition}

Observe that a poset $X$ is up-complete (resp. down-complete) if it is up-complete (resp. down-complete) as a topologized poset endowed with the anti-discrete topology $\{\emptyset,X\}$. On the other hand, a poset $X$ endowed with the discrete topology is complete if and only if $X$ is {\em chain-finite} in the sense that each chain in $X$ is finite.

Now we show that up-complete and down-complete topologized posets can be equivalently defined using chains instead of up-directed and down-directed subsets.
The following lemma is a topologized version of Iwamura's Lemma \cite{Iwamura}  (cf. \cite{Bruns}, \cite{Mark}).

%The proofs of the following two lemmas can be found in \cite{BBR}.

\begin{lemma}\label{l:up-p} Let $\kappa$ be an infinite cardinal. A topologized poset $X$ is ${\uparrow}\kappa$-complete if and only if any nonempty chain $C\subseteq X$ of cardinality $|C|\le\kappa$ has $\sup C\in\bar C$.
\end{lemma}

\begin{proof} The ``only if'' part is trivial as each chain is an up-directed subset in $X$. The ``if'' part will be proved by transfinite induction.

First we prove the lemma for the cardinal $\w$. Assume that each countable chain $C\subseteq X$ has $\sup C\in \bar C$. To prove that $X$ is ${\uparrow}\w$-complete,  take any nonempty countable up-directed subset $D=\{x_n\}_{n\in\w}$ in $X$. Put $y_0:=x_0$ and for every $n\in\IN$ choose an element $y_n\in D$ such that $y_n\ge x_n$ and $y_n\ge y_{n-1}$ (such an element $y_n$ exists as $D$ is up-directed).

By our assumption, the chain $C:=\{y_n\}_{n\in\w}$ has $\sup C\in \bar C\subset\bar D$. We claim that $\sup C$ is the least upper bound for the set $D$. Indeed, for any $n\in\w$ we get $x_n\le y_n\le\sup D$ and hence $\sup C$ is an upper bound for the set $D$. On the other hand, each upper bound $b$ for $D$ is an upper bound for $C$ and hence $\sup C\le b$. Therefore $\sup D=\sup C\in\bar C\subset\bar D$.
\smallskip

Now assume that for some uncountable cardinal $\kappa$  we have proved that each up-directed set $D$ of cardinality $|D|<\kappa$ in $X$ has $\sup D\in\bar D$ if each nonempty chain $C\subseteq X$ of cardinality $|C|<\kappa$ has $\sup C\in\bar C$. Assume that each chain $C\subseteq X$ of cardinality $|C|\le\kappa$ has $\sup C\in\bar C$. To prove that the topologized poset $X$ is a ${\uparrow}\kappa$-complete, fix any up-directed subset $D\subseteq X$ of cardinality $|D|\le\kappa$.

Since $D$ is up-directed, there exists a function $f:D\times D\to D$ assigning to any pair $(x,y)\in X\times X$ a point $f(x,y)\in D$ such that $x\le f(x,y)$ and $y\le f(x,y)$. Given any subset $A\subseteq D$, consider the increasing sequence $(A_n)_{n\in\w}$ of subsets of $A$, defined by the recursive formula $A_0:=A$ and $A_{n+1}:=A_n\cup f(A_n\times A_n)$ for $n\in\w$.
Finally, let $\langle A\rangle:=\bigcup_{n\in\w}A_n$ and observe that $\langle A\rangle$ is an up-directed subset of $D$. By induction it can be proved that $|A_n|\le \max\{\w,|A|\}$ for every $n\in\w$, and hence $|\langle A\rangle|\le\max\{\w,|A|\}$. Moreover, it can be also shown that for any subsets $A\subseteq B$ of $D$ we have $\langle A\rangle\subseteq \langle B\rangle$.

Write the set $D$ as $D=\{x_\alpha\}_{\alpha\in\kappa}$. For every $\beta\in\kappa$ let $D_\beta:=\langle\{x_\alpha\}_{\alpha\le\beta}\rangle$ and observe that $D_\beta\subseteq D$ is an up-directed set of cardinality $|D_\beta|\le\max\{\w,|\beta|\}<\kappa$. By the inductive assumption,  the up-directed set $D_\beta$ has $\sup D_\beta\in\bar D_\beta\subset\bar D$. For any ordinals $\alpha<\beta$ in $\kappa$ the inclusion $D_\alpha\subseteq D_\beta$ implies that $\sup D_\alpha\le \sup D_\beta$. By our assumption, the chain $C:=\{\sup D_\alpha:\alpha\in\kappa\}\subseteq \bar D$ has $\sup C\in \bar C\subseteq \bar D$. It remains to observe that $\sup C=\sup D$.
\end{proof}

\begin{corollary}\label{c:up-p}  A topologized poset $X$ is up-complete if and only if any nonempty chain $C\subseteq X$ has $\sup C\in\bar C$ in $X$.
\end{corollary}

Since a poset is up-complete if and only if it is up-complete as a topologized poset with the anti-discrete topology, Corollary~\ref{c:up-p} implies the following known characterization of up-complete posets, see \cite{Iwamura}, \cite{Bruns}, \cite{Mark}.

\begin{corollary}\label{c:up-poset}  A poset $X$ is up-complete if and only if any nonempty chain $C\subseteq X$ has $\sup C$ in $X$.
\end{corollary}

 We shall say that an up-directed set $D$ in a topologized poset $X$ {\em up-converges} to a point $x\in X$ if for each neighborhood $U_x\subseteq X$ of $x$ there exists $d\in D$ such that $D\cap{\uparrow}d\subseteq U_x$.

\begin{lemma}\label{l:up-conv} Let $\kappa$ be an infinite cardinal. If a topologized poset $X$ is ${\uparrow}\kappa$-complete, then each up-directed set $D\subseteq X$ of cardinality $|D|\le\kappa$ up-converges to its $\sup D$.
\end{lemma}

\begin{proof} To derive a contradiction, assume that some up-directed set $D\subseteq X$ of cardinality $|D|\le\kappa$ does not converge to $\sup D$. Then there exists an open neighborhood $U\subseteq X$ of $\sup D$ such that for every $d\in D$ the set $D\cap {\uparrow}d\setminus U$ is not empty. Then the set $E:=D\setminus U$ is directed and cofinal in $D$. To see that $E$ is directed, take any points $x,y\in E\subseteq D$ and find a point $z\in D$ with $x\le z$ and $y\le z$. Since $E$ is cofinal in $D$, there exists $e\in E$ such that $z\le e$. Then $x\le e$ and $y\le e$ by the transitivity of the partial order. By the up-completeness of $X$, the directed set $E$ has $\sup E\in \bar E$. The inclusion $E\subseteq D$ implies that $\sup E\le \sup D$ and the cofinality of $E$ in $D$ that $\sup E=\sup D$. Taking into account that $E\cap U=\emptyset$, we obtain a contradiction: $\sup D=\sup E\in\bar E\cap U=\emptyset$, completing  the proof.
\end{proof}

We shall say that a down-directed set $D$ in a topologized poset $X$ {\em down-converges} to a point $x\in X$ if for each neighborhood $U_x\subseteq X$ of $x$ there exists $d\in D$ such that $D\cap{\downarrow}d\subseteq U_x$.

 Applying Lemmas~\ref{l:up-p}, \ref{l:up-conv} to the opposite partial order on a topologized poset, we obtain the following dual versions of these lemmas.

\begin{lemma}\label{l:down-p} Let $\kappa$ be an infinite cardinal. A topologized poset $X$ is ${\downarrow}\kappa$-complete if and only if any nonempty chain $C\subseteq X$ of cardinality $|C|\le\kappa$ has $\inf C\in\bar C$.
\end{lemma}

\begin{lemma}\label{down-conv} Let $\kappa$ be an infinite cardinal. If a topologized poset $X$ is ${\downarrow}\kappa$-complete, then each down-directed set $D\subseteq X$ of cardinality $|D|\le\kappa$ down-converges to its $\inf D$.
\end{lemma}

Lemma~\ref{l:down-p} implies the following two characterizations.

\begin{corollary}\label{c:down-p}  A topologized poset $X$ is down-complete if and only if any nonempty chain $C\subseteq X$ has $\inf C\in\bar C$ in $X$.
\end{corollary}

\begin{corollary}\label{c:down-poset}  A poset $X$ is down-complete if and only if any nonempty chain $C\subseteq X$ has $\inf C$ in $X$.
\end{corollary}

Unifying Corollaries~\ref{c:up-p} and \ref{c:down-p} we obtain the following useful characterization of completeness of topologized posets.

\begin{theorem}\label{t:completeT} A topologized poset $X$ is complete if and only if each nonempty chain $C\subseteq X$ has $\sup C\in\bar C$ and $\inf C\in\bar C$.
\end{theorem}

Since a poset is complete if and only if it is complete as a topologized poset with the anti-discrete topology, Corollary~\ref{t:completeT} implies the following (known) characterization of complete posets.

\begin{theorem}\label{t:complete} A poset $X$ is complete if and only if each nonempty chain $C\subseteq X$ has $\sup C$ and $\inf C$ in $X$.
\end{theorem}

\section{Preserving the completeness by Tychonoff products of topologized posets}\label{s:oper}

Now we prove that the completeness of topologized posets is preserved by Tychonoff products. On the Tychonoff product $\prod_{\alpha\in A}X_\alpha$ of topologized posets $(X_\alpha,\le_\alpha)$ we consider the pointwise partial order $\le$ defined by $(x_\alpha)_{\alpha\in A}\le(y_\alpha)_{\alpha\in A}$ iff $x_\alpha\le_\alpha y_\alpha$ for each $\alpha\in A$.

\begin{theorem}\label{t:up-prod} Let $\kappa$ be an infinite cardinal. The Tychonoff product $X:=\prod_{\alpha\in A}X_\alpha$ of ${\uparrow}\kappa$-complete topologized posets $X_\alpha$, $\alpha\in A$, is an ${\uparrow}\kappa$-complete topologized poset.
\end{theorem}

\begin{proof} 
By Lemma~\ref{l:up-p}, the ${\uparrow}\kappa$-completeness of $X$ will follow as soon as we prove that each nonempty chain $C\subseteq X$ of cardinality $|C|\le\kappa$ has $\sup C\in \bar C$. For every $\alpha\in A$ let $\pr_\alpha:X\to X_\alpha$ denote the coordinate projection. Since the projection $\pr_\alpha$ is monotone, the image $C_\alpha:=\pr_\alpha(C)$ of the chain $C$ is a chain in $X_\alpha$ of cardinality $|C_\alpha|\le|C|\le\kappa$. By the ${\uparrow}\kappa$-completeness of the topologized poset $X_\alpha$, the chain $C_\alpha$ has $\sup C_\alpha\in \bar C_\alpha$.

Consider the elements $c:=(\sup C_\alpha)_{\alpha\in A}\in X$. It is clear that $c$ is an upper bound for $C$. We claim that $c=\sup C$. Indeed, given any other upper bound $b=(b_\alpha)_{\alpha\in A}\in X$ of $C$, we have $\sup C_\alpha\le b_\alpha$ for all $\alpha\in A$ and hence $c\le b$. So, $c=\sup C$.

It remains to show that $c\in\bar C$. Assuming that $c\notin\bar C$, we could find an open neighborhood $O_c\subseteq X$ such that $O_c\cap C=\emptyset$. Replacing $O_c$ by a smaller neighborhood, we can assume that $O_c$ is of the basic form $O_c=\prod_{\alpha\in A}U_\alpha$, where the set $F=\{\alpha\in A:U_\alpha\ne X_\alpha\}$ is finite. For every $\alpha\in F$ the set $U_\alpha$ is a neighborhood of $\sup C_\alpha$. Applying Lemma~\ref{l:up-conv}, find a point $u_\alpha\in C_\alpha$ such that $C_\alpha\cap{\uparrow}u_\alpha\subseteq U_\alpha$. Since $u_\alpha\in C_\alpha=\pr_\alpha(C)$, there exists an element $v_\alpha\in C$ such that $\pr_\alpha(v_\alpha)=u_\alpha$. Since $C$ is a chain, the finite set  $\{v_\alpha:\alpha\in F\}$ has a largest element $v$. For this element we have $\pr_\alpha(v)\in C_\alpha\cap{\uparrow}u_\alpha\subseteq U_\alpha$ for all $\alpha\in F$. Consequently, $v\in C\cap O_c$, which contradicts the choice of the neighborhood $O_c$.
\end{proof}

Applying Theorem~\ref{t:up-prod} to the opposite order on a topologized poset, we get the following dual version of Theorem~\ref{t:up-prod}.

\begin{theorem}\label{t:down-prod} Let $\kappa$ be an infinite cardinal. The Tychonoff product $X:=\prod_{\alpha\in A}X_\alpha$ of ${\downarrow}\kappa$-complete topologized posets $X_\alpha$, $\alpha\in A$, is a ${\downarrow}\kappa$-complete topologized poset.
\end{theorem}

Theorems~\ref{t:up-prod} and \ref{t:down-prod} imply:

\begin{theorem}\label{t:updown-prod} Let $\kappa$ be an infinite cardinal. The Tychonoff product $X:=\prod_{\alpha\in A}X_\alpha$ of ${\updownarrow}\kappa$-complete topologized posets $X_\alpha$, $\alpha\in A$, is an ${\updownarrow}\kappa$-complete topologized poset.
\end{theorem}

Let us also note the following obvious preservation property of complete topologized posets.

\begin{proposition}\label{p:sub} If a topologized poset $X$ is complete (resp. up-complete, down-complete), then so is each closed topologized subposet in $X$.
\end{proposition}

\section{Interplay between completeness and chain-compactness}\label{s:cc}

In this section we establish the relation between the completeness and chain-compactness of topologized posets.

A topologized poset is defined to be {\em chain-compact} if each closed chain in $X$ is compact.

\begin{lemma}\label{l:c=>cc} Each complete topologized poset $X$ is chain-compact.
\end{lemma}

\begin{proof} Given a nonempty closed chain $C$ in a complete topologized poset $X$, we shall prove that $C$ is compact. Given any open cover $\U$ of $C$, we should find a finite subfamily $\U'\subseteq \U$ such that $C\subset\bigcup\U'$.  By the down-completeness of $X$, the closed chain $C$ has $\inf C\in\bar C=C$, which means that $c:=\inf C$ is the smallest element of $C$. Consider the set $A\subseteq C$ consisting of points $a\in C$ such that the closed interval $[c,a]:=\{x\in C:c\le x\le a\}$ can be covered by a finite subfamily of the cover $\U$.
The set $A$ contains the point $c$ and hence is not empty. By the up-completeness of $X$, the set $A$ has $\sup A\in\bar A\subset\bar C$.

We claim that the point $b:=\sup A$ belongs to $A$. To derive a contradiction, assume that $b\notin A$. Choose any open set $U_b\in \U$ containing the point $b:=\sup A\subseteq \bar C$.
By Lemma~\ref{l:up-conv}, the chain $A$ contains a point $a\in A$ such that $A\cap{\uparrow}a\subseteq U_b$. Taking into account that for any $x\in A$ the interval $[c,x]\subseteq \bar C$ is contained in $A$, we see that $A\cap{\uparrow}a=[a,b]\setminus\{b\}$. Then $[a,b]=(A\cap{\uparrow}a)\cup\{b\}\subseteq U_b$.
Now the definition of the set $A\ni a$ yields a finite subfamily $\V\subset\U$ such that $[c,a]\subset\bigcup\V$. Then for the finite  subfamily $\U'=\V\cup\{U_b\}$ of $\U$ we have $[c,b]=[c,a]\cup[a,b]\subset\bigcup \U'$, which means that $b\in A$.

To complete the proof, it suffices to show that $C\subseteq \bigcup\U'$. Assuming that the (closed) set $E:=C\setminus\bigcup\U'$ is not empty, we can apply the completeness of $X$ and find $\inf E\in \bar E=E$. Choose any open set $U_e\in\U$ containing the point $e:=\inf E$ and observe that $\U'\cup\{U_e\}$ is a finite subfamily covering the set $[c,e]$, which means that $e\in A$. On the other hand, the (non)inclusion $[c,b]\subset\bigcup\U'\not\ni e$ implies that $b<e$, which contradicts the equality $b=\sup A$. \end{proof}

Lemma~\ref{l:c=>cc} can be reversed for ${\uparrow}{\downarrow}$-closed topologized posets. We define a topologized poset $X$ to be
\begin{itemize}
\item {\em ${\uparrow}$-closed} if the upper set ${\uparrow}x$ of any point $x\in X$ is closed in $X$;
\item {\em ${\downarrow}$-closed} if the lower set ${\downarrow}x$ of any point $x\in X$ is closed in $X$;
\item {\em ${\downarrow}{\uparrow}$-closed} if it is ${\uparrow}$-closed and ${\downarrow}$-closed;
\item {\em ${\updownarrow}$-closed} if for any point $x\in X$ the set ${\updownarrow}x:={\uparrow}x\cup{\downarrow}x$ is closed in $X$;
\item a {\em pospace} if the partial order $\le$ is a closed subset of $X\times X$;
\item {\em chain-closed} if the closure of each chain in $X$ is a chain.
\end{itemize}
For topologized posets we have the following implications:
$$\mbox{pospace} \Ra \mbox{${\uparrow}{\downarrow}$-closed} \Ra \mbox{${\updownarrow}$-closed} \Ra \mbox{chain-closed}.$$
The last implication is not entirely trivial and is proved in the following lemma.

\begin{lemma}\label{l:cc1} Each ${\updownarrow}$-closed topologized poset is chain-closed.
\end{lemma}

\begin{proof}  Given a chain $C\subseteq X$, we should prove that its closure $\bar C$ in $X$ is a chain. Assuming that $\bar C$ contains two incomparable elements $x$  and $y$, observe that $V_x:=X\setminus {\updownarrow}y$ is an open neighborhood of $x$. Since $x\in\bar C$, there exists an element $z\in C\cap V_x$. It follows from $z\notin{\updownarrow}y$ that $V_y=X\setminus {\updownarrow}z \subseteq X\setminus C$ is an open neighborhood of $y$, disjoint with $C$, which is not possible as $y\in\bar C$.
\end{proof}

%By Theorem 3.1 in \cite{BBm}, a Hausdorff semi-topological semilattice is chain-compact if and only if it is complete. In the following lemma we show that this characterization holds in the more general framework of ${\uparrow}{\downarrow}$-closed topologized posets.

\begin{theorem}\label{t:cc<=>kc} An ${\uparrow}{\downarrow}$-closed topologized poset $X$ is complete if and only if it is chain-compact.
\end{theorem}

\begin{proof} The ``only if'' part follows from Lemma~\ref{l:c=>cc}. To prove the ``if'' part, assume that an ${\uparrow}{\downarrow}$-closed topologized poset $X$ in chain-compact. To prove that $X$ is complete, take any nonempty chain $C$. By Lemma~\ref{l:cc1}, the closure $\bar C$ of $C$ is a chain in $X$. By the chain-compactness of $X$, the closed chain $\bar C$ is compact. By the compactness of $\bar C$, the centered family $\F=\{\bar C\cap{\downarrow}x:x\in \bar C\}$ of closed subsets of $\bar C$ has nonempty intersection, which is a singleton, containing the smallest element $s$ of the compact chain $\bar C$. It is clear that $s$ is a lower bound for the set $C$. On the other hand, for any other lower bound $b$ for the set $C$, we get $C\subset{\uparrow}b$ and hence $s\in\bar C\subset\overline{{\uparrow}b}={\uparrow}b$ and finally, $b\le s$. So, $s=\inf C$.

By analogy we can prove that the compact chain $\bar C$ has the largest element which coincides with $\sup C$.
\end{proof}

\section{Complete topologized semilattices}\label{s:5}
In this section we study the notion of completeness in the framework of topologized semilattices.

By a {\em semilattice} we understand a commutative semigroup $X$ of idempotents (the latter means that $xx=x$ for all $x\in X$). Each semilattice $X$ carries a natural partial order $\le$ defined by $x\le y$ iff $xy=x$. So, we can consider a semilattice $X$ as a poset such that each nonempty finite subset $A=\{a_1,\dots,a_n\}\subseteq X$ has $\inf A=a_1\cdots a_n$.

By a {\em topologized semilattice} we understand a semilattice endowed with a topology. A topologized semilattice $X$ is called a
({\em semi\/}){\em topological semilattice} if the semilattice operation $X\times X\to X$, $(x,y)\mapsto xy=\inf\{x,y\}$, is (separately) continuous.

A topologized semilattice is {\em complete} (resp. {\em up-complete, down-complete}) if it is complete (resp. up-complete, down-complete) as a topologized poset endowed with the natural order, induced by the semilattice operation. By Theorem~\ref{t:completeT}, a topologized semilattice $X$ is {\em complete} if and only if each nonempty chain $C\subseteq X$ has $\inf C\in\bar C$ and $\sup C\in\bar C$. Observe that a discrete topological semilattice is complete if and only if it is {\em chain-finite} in the sense that each chain in $X$ is finite.

The completeness of topologized semilattices is preserved by many operations. The following two propositions are partial cases of Proposition~\ref{p:sub} and Theorem~\ref{t:updown-prod}.

\begin{proposition} Let $X$ be a closed subsemilattice of a topologized semilattice $Y$. If $Y$ is complete (resp. up-complete, down-complete), then so is the topologized semilattice $X$.
\end{proposition}

\begin{proposition}  If topologized semilattices $X_\alpha$, $\alpha\in A$, are complete (resp. up-complete, down-complete), then so is their Tychonoff product $\prod_{\alpha\in A}X_\alpha$.
\end{proposition}

A topologized poset $X$ is defined to be {\em weakly ${\uparrow}$-closed} if for every $x\in X$ we have $\overline{\{x\}}\subset{\uparrow}x$. It is easy to see that a topologized poset is weakly ${\uparrow}$-closed if it is ${\uparrow}$-closed or satisfies the separation axiom $T_1$.

\begin{lemma}\label{l:homo} Let $\kappa$ be an infinite cardinal. Let $h:X\to Y$ be a continuous surjective homomorphism from a topologized semilattice $X$ to a weakly ${\uparrow}$-closed topologized semilattice $Y$. If the topologized semilattice $X$ is down-complete (and ${\uparrow}\kappa$-complete), then so is the topologized semilattice $Y$.
\end{lemma}

\begin{proof} Assume that the topologized semilattice $X$ is down-complete (and ${\uparrow}\kappa$-complete). By Corollary~\ref{c:down-p} (and Lemma~\ref{l:up-p}), the down-completeness (and ${\uparrow}\kappa$-completeness) of $X$ will follow as soon as we show that each nonempty chain $C\subseteq X$ (of cardinality $|C|\le\kappa$) has $\inf C\in\bar C$ (and $\sup C\in\bar C$).

Observe that for every $c\in C$ the preimage $h^{-1}(c)$ is a subsemilattice in $X$ and hence is a down-directed set in $X$. By the down-completeness of $X$, it has $\inf h^{-1}(c)\in\overline{h^{-1}(c)}$. Let $b_c:=\inf h^{-1}(c)$.
By the continuity of $h$ and the weak $\uparrow$-closedness of $Y$, $$h(b_c)\in h(\overline{h^{-1}(c)})\subset\overline{h(h^{-1}(c))}=\overline{\{c\}}\subseteq {\uparrow}c$$ and hence $c\le h(b_c)$. On the other hand, for any $x\in h^{-1}(c)$, we get $b_c\le x$ and hence $h(b_c)\le h(x)=c$ and finally $$h(b_c)=c.$$ Let us show that $b_c=\inf h^{-1}({\uparrow}c)$. Indeed, for any $x\in h^{-1}({\uparrow}c)$ we get $h(x\cdot b_c)=h(x)\cdot h(b_c)=h(x)\cdot c=c$ and hence $b_c\le x\cdot b_c\le x$. So, $b_c$ is a lower bound for the set $h^{-1}({\uparrow}c)$. On the other hand, any lower bound $b$ of the set $h^{-1}({\uparrow}c)$ is a lower bound of the set $h^{-1}(c)\subseteq h^{-1}({\uparrow}c)$ and hence $b\le b_c$, which means that $b_c=\inf h^{-1}({\uparrow}c)$.

For every elements $x\le y$ in $C$ the inclusion $h^{-1}({\uparrow}y)\subseteq h^{-1}({\uparrow}x)$ implies $b_x=\inf h^{-1}({\uparrow}x)\le\inf h^{-1}({\uparrow}y)=b_y$. So the set $D=\{b_c:c\in C\}$ is a chain (of cardinality $\le\kappa$).
By the down-completeness (and ${\uparrow}\kappa$-completeness) of $X$ the chain $D$ has $\inf D\in\bar D$ (and $\sup D\in\bar D$). The continuity of $h$ implies that $f(\inf D)\in f(\bar D)\subset\overline{f(D)}=\bar C$ (and
$f(\sup D)\in f(\bar D)\subset\overline{f(D)}=\bar C$).

It remains to prove that $f(\inf D)=\inf C$ (and $f(\sup D)=\sup C$). The monotonicity of $f$ implies that $f(\inf D)$ is a lower bound (and $f(\sup D)$ is an upper bound) for $f(D)=C$. For any lower bound $\lambda$ of $C$, we get $C\subseteq {\uparrow}\lambda$ and $D\subseteq h^{-1}(C)\subseteq h^{-1}({\uparrow}\lambda)$. Then $b_\lambda=\inf h^{-1}({\uparrow}\lambda)\le \inf D$ and $\lambda=h(b_\lambda)\le h(\inf D)$, which means that $h(\inf D)=\inf C$. (For any upper bound $u$ of $C$ and any $c\in C$ the inequality $c\le u$ implies $b_c\le b_u$. So, $b_u$ is an upper bound of the chain $D=\{b_c:c\in C\}$. Then $\sup D\le b_u$ and $h(\sup D)\le h(b_u)=u$ by the monotonicity of $h$, witnessing that $h(\sup D)=\sup C$).
\end{proof}

Finally, we discuss the relation of the completeness of a topologized semilattice $X$ to the compactness of the the {\em weak$^\bullet$ topology} $\Zar_X$, which is generated by the subbase consisting of complements to closed subsemilattices in $X$. A topologized semilattice $X$ is called {\em $\Zar$-compact} if its weak$^\bullet$ topology $\Zar_X$ is compact. The weak$^\bullet$-topology was introduced and studied in \cite{BBw}. According to Lemmas 5.4, 5.5 of \cite{BBw}, for any topologized semilattice we have the implications:
$$\mbox{complete} \Ra \mbox{$\Zar$-compact} \Ra \mbox{chain-compact}.
$$These implications combined with Theorem~\ref{t:cc<=>kc} yield the following characterization.

\begin{theorem} For an ${\uparrow}{\downarrow}$-closed topologized semilattice $X$ the following conditions are equivalent:
\begin{enumerate}
\item $X$ is complete;
\item$X$ is $\Zar_X$-compact;
\item $X$ is chain-compact.
\end{enumerate}
\end{theorem}

\section{Absolute closedness of complete topologized semilattices}\label{s:ac}

Quite often the notion of completeness is connected with the absolute closedness, understood in an appropriate sense, see e.g. \cite{Ban}, \cite{BBc}, \cite{B-G-2012}, \cite{B-G-R}, \cite{D-U-1998}, \cite{Gutik-2014}, \cite{G-P-R-2010},  \cite{G-P-2001}, \cite{G-R-2008}, \cite{Raikov1946}, \cite{R-2003}. For example, a metric space $X$ is complete if and only if it is closed in each metric space $Y$, containing $X$ as a metric subspace. A uniform space $X$ is complete if and only if $X$ is closed in each uniform space, containing $X$ as a uniform subspace. A topological group $X$ is complete in its two-sided uniformity if and only if $X$ is closed in any topological group, containing $X$ as a topological subgroup. A similar phenomenon happens in the category of (semi)topological semilattices.

Historically the first result in this direction belongs to J.W.~Stepp \cite{Stepp1969}, \cite{Stepp1975} who
 proved that any chain-finite semilattice $X$ is closed in each Hausdorff topological semilattice, containing $X$ as a subsemilattice. This result of Stepp was improved to the following characterizations, which can be found in \cite{BBm}.

\begin{theorem} A discrete topological semilattice $X$ is chain-finite if and only if $X$ is closed in any Hausdorff zero-dimensional topological semilattice $Y$ containing $X$ as a subsemilattice.
\end{theorem}

\begin{theorem} A Hausdorff topological semilattice $X$ is complete if and only if each closed subsemilattice $Z$ of $X$ is closed in any Hausdorff topological semilattice containing $Z$ as a topological subsemilattice.
\end{theorem}

The ``only if'' parts of these characterizations are partial cases of the following theorems on closedness of (chain-finite) complete topologized semilattices under continuous homomorphisms.

\begin{theorem}\label{t:cf} For any  homomorphism $h:X\to Y$ from a chain-finite semilattice $X$ to a Hausdorff  semitopological semilattice $Y$, the image $h(X)$ is closed in $Y$.
\end{theorem}

\begin{theorem}\label{t:ct} For any continuous homomorphism $h:X\to Y$ from a complete topologized semilattice $X$ to a Hausdorff topological semilattice $Y$, the image $h(X)$ is closed in $Y$.
\end{theorem}

In fact, Theorems~\ref{t:cf} and \ref{t:ct} are corollaries of more general results related to upper semicontinuous $T_i$-multimorphisms between topologized semilattices.

A {\em multi-valued map} $\Phi:X\setmap Y$ between sets $X,Y$ is a function assigning to each point $x\in X$ a subset $\Phi(x)$ of $Y$. For a subset $A\subseteq X$ we put $\Phi(A):=\bigcup_{x\in A}\Phi(x)$. A multi-valued map $\Phi:X\setmap Y$ between semigroups is called a {\em multimorphism} if $\Phi(x)\cdot\Phi(y)\subset\Phi(xy)$ for any $x,y\in X$. Here $\Phi(x)\cdot\Phi(y):=\{ab:a\in\Phi(x),\;b\in\Phi(y)\}$.

A multi-valued map $\Phi:X\setmap Y$ between topological spaces is called
{\em upper semicontinuous} if for any closed subset $F\subseteq Y$ the preimage $\Phi^{-1}(F):=\{x\in X:\Phi(x)\cap F\ne\emptyset\}$ is closed in $X$.

A subset $F$ of a topological space $X$ is called {\em $T_1$-closed} (resp.
{\em $T_2$-closed\/}) in $X$ if each point $x\in X\setminus F$ has a (closed) neighborhood, disjoint with $F$.

A multimorphism $\Phi:X\setmap Y$ is called a {\em $T_i$-multimorphism} for $i\in\{1,2\}$ if for any $x\in X$ the set $\Phi(x)$ is $T_i$-closed in $Y$.

The following two theorems (implying Theorems~\ref{t:cf} and \ref{t:ct}) are proved in \cite{BBm}.

\begin{theorem} For any $T_1$-multimorphism $\Phi:X\setmap Y$ from a chain-finite semilattice $X$ to a semitopological semilattice $Y$, the image $\Phi(X)$ is closed in $Y$.
\end{theorem}

\begin{theorem}\label{t:multi} For any upper semi-continuous $T_2$-multimorphism $\Phi:X\setmap Y$ from a complete topologized semilattice $X$ to a topological semilattice $Y$, the image $\Phi(X)$ is closed in $Y$.
\end{theorem}

Looking at Theorems~\ref{t:cf} and \ref{t:ct}, one can ask the following problem.

\begin{problem}\label{prob:abscl} Let $h:X\to Y$ be a continuous homomorphism from a complete topologized semilattice $X$ to a Hausdorff semitopological semilattice $Y$. Is the set $h(X)$ closed in $Y$?
\end{problem}

%Observing that each Hausdorff topological semilattice is a pospace, we can ask the following problem.

%\begin{problem}\label{prob:pospace} Is the partial order of any (complete)  Hausdorff semitopological semilattice $X$ closed in $X\times X$?
%\end{problem}

%\begin{lemma}\label{l:down} Let $\kappa$ be an infinite cardinal. If a topologized semilattice $X$ is ${\downarrow}\kappa$-complete, then each nonempty subset $A\subseteq X$ of cardinality $\kappa$ has $\inf A\in \bar S$ where $S$ is the semilattice generated by $A$ in $X$.
%\end{lemma}

%\begin{proof} By Lemma~\ref{l:down-p}, the semilattice $S$ has $\inf S\in\bar S$ (being a down-directed subset of $X$). Since $A\subseteq S$, the element $\inf S$ is a lower bound for $A$ in $X$. On the other hand, any lower bound $b$ for the set $A$ is a lower bound for the semilattice $S$ generated by $A$, which implies that $b\le \inf S$ and means that $\inf S$ is the greatest lower bound $\inf A$ of the set $A$ in in $X$.
%\end{proof}
%Lemma~\ref{l:down2} implies

Problem~\ref{prob:abscl} has affirmative answer for homomorphisms to sequential  semitopological semilattices.

We recall that a topological space $X$ is {\em sequential} if each sequentially closed subset in $X$ is closed. A subset $A$ of a topological space $X$ is called {\em sequentially closed} if $A$ contains the limit points of all sequences $\{a_n\}_{n\in\w}\subseteq A$ that converge in $X$.

A topological space $X$ is {\em countably tight} if for any subset $A\subseteq X$ and point $a\in\bar A$ there exists a countable subset $B\subseteq A$ such that $a\in\bar B$. It is well-known \cite[1.7.13(c)]{Engelking1989} that each subspace of a sequential topological space has countable tightness. The following (non-trivial) results were proved in \cite{BBs}.

\begin{theorem}\label{t:seq} For any continuous homomorphism $h:X\to Y$ from a countably tight complete topologized semilattice $X$ to a Hausdorff semitopological semilattice $Y$, the image $h(X)$ is sequentially closed in $Y$.
% and the partial order $\{(x,y)\in X\times X:x\le y\}$ of the semilattice $X$ is sequentially closed in $X\times X$.
\end{theorem}

\begin{corollary}\label{c:seq} For any continuous homomorphism $h:X\to Y$ from a complete topologized semilattice $X$ to a sequential Hausdorff semitopological semilattice $Y$, the image $h(X)$ is closed in $Y$.
% and the partial order $\{(x,y)\in X\times X:x\le y\}$ of the semilattice $X$ is sequentially closed in $X\times X$.
\end{corollary}

Also Problem~\ref{prob:abscl} has an affirmative answer for semitopological semilattices satisfying the separation axiom $\vv{T}_{2\delta}$ (which is stronger than the Hausdorff axiom $T_2$ but weaker than axiom $\vv{T}_{3\frac12}$ of the functional Hausdorffness).

Let us recall that a topological space $X$ satisfies the separation axiom
\begin{itemize}
\item $T_1$ if for any distinct points $x,y\in X$ there exists an open set $U\subseteq X$ such that $x\in U\subseteq X\setminus\{y\}$;
\item $T_2$ if for any distinct points $x,y\in X$ there exists an open set $U\subseteq X$ such that $x\in U\subset\bar U\subseteq X\setminus\{y\}$;
\item $T_3$ if $X$ is a $T_1$-space and for any open set $V\subseteq X$ and point $x\in V$ there exists an open set $U\subseteq X$ such that $x\in U\subset\bar U\subseteq V$;
\item $T_{3\frac12}$ if $X$ is a $T_1$-space and for any open set $V\subseteq X$ and point $x\in V$ there exists a continuous function $f:X\to[0,1]$ such that $x\in f^{-1}([0,1))\subseteq V$;
\item $T_{2\delta}$ if $X$ is a $T_1$-space and for any open set $V\subseteq X$ and point $x\in V$ there exists a countable family $\U$ of closed neighborhoods of $x$ in $X$ such that $\bigcap\U\subseteq V$;
\item $\vv{T}_i$ for $i\in\{1,2,2\delta,3,3\frac12\}$ if $X$ admits a bijective continuous map $X\to Y$ to a $T_i$-space $Y$.
\end{itemize}
%In these definitions by $\bar A$ be denote the closure of a subset $A$ in a topological space.
Topological spaces satisfying a separation axiom $T_i$ are called {\em $T_i$-spaces}. The separation axioms $T_{2\delta}$ and $\vv{T}_{2\delta}$ were introduced in \cite{BBR}.

The following diagram describes the implications between the separation axioms $T_i$ and $\vv{T}_i$ for $i\in\{1,2,2\delta,3,3\frac12\}$.
$$\xymatrix{
T_{3\frac12}\ar@{=>}[r]\ar@{=>}[d]&T_3\ar@{=>}[r]\ar@{=>}[d]&T_{2\delta}\ar@{=>}[r]\ar@{=>}[d]&T_2\ar@{<=>}[d]\ar@{=>}[r]&T_1\ar@{<=>}[d]\\
\vv{T}_{3\frac12}\ar@{=>}[r]&\vv{T}_3\ar@{=>}[r]&\vv{T}_{2\delta}\ar@{=>}[r]&\vv{T}_2\ar@{=>}[r]&T_1
}
$$
Observe that a topological space $X$ satisfies the separation axiom $\vv{T}_{3\frac12}$ if and only if it is {\em functionally Hausdorff\/} in the sense that for any distinct points $x,y\in X$ there exists a continuous function $f:X\to\mathbb R$ with $f(x)\ne f(y)$. Therefore, each functionally Hausdorff space is a $\vv{T}_{2\delta}$-space. An example of a Hausdorff space which is not $\vv{T}_{2\delta}$ was constructed in \cite{BBR}.

The following (non-trivial) result was proved in \cite{BBR}.

\begin{theorem}\label{t:Tdelta} For any continuous homomorphism $h:X\to Y$ from a complete topologized semilattice to a semitopological semilattice $Y$ satisfying the separation axiom $\vv{T}_{2\delta}$, the image $h(X)$ is closed in $Y$.% the partial order $\{(x,y)\in X\times X:x\le y\}$ of the semilattice $X$ is closed in $Y\times Y$ and $X$ is closed in $Y$.
\end{theorem}

In \cite{BBG} it was proved that the answer to Problem~\ref{prob:abscl} is affirmative  for homomorphisms into $\w$-Lawson semitopological semilattices.

A topologized semilattice  is called {\em Lawson} (see \cite[p.12]{CHK})  it has a base of the topology consisting of open subsemilattices.

For a Hausdorff topologized semilattice $X$ its {\em Lawson number} $\bar\Lambda(X)$ is defined as the smallest cardinal $\kappa$ such that for any distinct points $x,y$ in $X$ there exists a family $\U$ of closed neighborhoods of $x$ in $X$ such that $|\U|\le\kappa$ and $\bigcap\U$ is a closed subsemilattice of $X$ which does not contain $y$. It is easy to see that $\bar\Lambda(X)\le\bar\psi(X)$, where $\bar\psi(X)$ is the smallest cardinal $\kappa$ such that for any point $x\in X$ there exists a family $\U$ of closed neighborhoods of $x$ in $X$ such that $|\U|\le\kappa$ and $\bigcap\U=\{x\}$.

A topologized semilattice $X$ is called {\em $\kappa$-Lawson} for some cardinal $\kappa$ if it is Hausdorff and $\bar\Lambda(X)\le\kappa$. The Lawson number of a Hausdorff topologized semilattice was introduced and studied in \cite{BBG}.

By \cite{BBG}, each Hausdorff Lawson semitopological semilattice is 1-Lawson. Moreover, a compact Hausdorff semitopological semilattice $X$ is Lawson if and only if it is 1-Lawson. Each Hausdorff topological semilattice $X$ is $\w$-Lawson. Each Hausdorff linear topologized semilattice is Lawson and 1-Lawson. The following (non-trivial) example was constructed in \cite{BBG}.

\begin{example} For any infinite cardinal $\lambda$ there exists a Hausdorff zero-dimensional semitopological semilattice $X$ such that $|X|=\lambda$ and $\bar\Lambda(X)=\bar\psi(X)=\mathrm{cf}(\lambda)$.
\end{example}
 
Our next example is a ``semilattice'' modification of Example 1 \cite{BBR} of a Hausdorff topological space that does not satisfy the separation axiom $\vec T_{2\delta}$.

\begin{example} There exists a Lawson Hausdorff topological semilattice that does not satisfy the separation axiom $\vec T_{2\delta}$.
\end{example}

\begin{proof} Consider the set $L=\{x_\alpha\}_{\alpha\le\w_1}\cup\{z\}\cup\{y_\alpha\}_{\alpha\le \w_1}$ of pairwise distinct points endowed with the linear order  in which $x_\alpha<x_\beta<z<y_\beta<y_\alpha$ for any  ordinals $\alpha<\beta\le\w_1$. Let $\ddot L:=L\setminus\{x_{\w_1},y_{\w_1}\}$. On the set $$X=(\ddot L\times[0,\w_1))\cup(\{x_{\w_1},y_{\w_1}\}\times\{\w_1\})$$ consider the semilattice operation
$$(x,\alpha)\cdot(y,\beta):=\begin{cases}
(\min\{x,y\},\min\{\alpha,\beta\})&\mbox{if $\alpha,\beta<\w_1$};\\
(\min\{x,z\},\alpha),&\mbox{if $\alpha<\w_1=\beta$};\\
(\min\{z,y\},\beta),&\mbox{if $\beta<\w_1=\alpha$};\\
(\min\{x,y\},\w_1),&\mbox{if $\alpha=\w_1=\beta$}.
\end{cases}
$$
Endow $X$ with the topology $\tau$ consisting of all sets $U\subseteq X$ satisfying the following three conditions:
\begin{itemize}
\item if $(z,\alpha)\in U$ for some $\alpha\in[0,\w_1)$, then $\{(x_\gamma,\alpha),(y_\gamma,\alpha):\gamma\in[\beta,\w_1)\}\subseteq U$ for some $\beta\in[0,\w_1)$;
\item if $(x_{\w_1},\w_1)\in U$, then $\{(x_\beta,\gamma):\beta,\gamma\in[\alpha,\w_1)\}\subseteq U$ for some $\alpha\in[0,\w_1)$;
\item if $(y_{\w_1},\w_1)\in U$, then $\{(y_\beta,\gamma):\beta,\gamma\in[\alpha,\w_1)\}\subseteq U$ for some $\alpha\in[0,\w_1)$.
\end{itemize}
It can be shown that $(X,\tau)$ is a Lawson (and hence 1-Lawson) Hausdorff topological semilattice, which does not satisfy the separation axiom $\vec T_{2\delta}$.
\end{proof}

The following partial answer to Problem~\ref{prob:abscl} was obtained in \cite{BBG}.

\begin{theorem}\label{t:Gd} For any continuous homomorphism $h:X\to Y$ from a complete topologized semilattice $X$ to an $\w$-Lawson semitopological semilattice $Y$, the image $h(X)$ is closed in $Y$.
\end{theorem} 

Corollary~\ref{c:seq} and Theorems~\ref{t:Tdelta}, \ref{t:Gd} imply the following partial answer to Problem~\ref{prob:abscl}.

\begin{corollary}\label{c:final} For a continuous homomorphism $h:X\to Y$ from a complete topologized semilattice $X$ to a Hausdorff semitopological semilattice $Y$, the image $h(X)$ is closed in $Y$ if one of the following conditions is satisfied:
\begin{enumerate}
\item $Y$ is a topological semilattice;
\item the topologized semilattice $Y$ is $\w$-Lawson;
\item the topological space $Y$ is sequential;
\item the topological space $Y$ satisfies the separation axiom $\vv{T}_{2\delta}$;
\item the topological space $Y$ is functionally Hausdorff.
\end{enumerate}
\end{corollary}

Comparing Theorems~\ref{t:Tdelta} and \ref{t:multi} we can ask the following problem.

\begin{problem} Let $\Phi:X\setmap Y$ be an upper semi-continuous $T_2$-multimorphism from a complete topologized semilattice $X$ to a semitopological semilattice $Y$ such that $Y$ is $\w$-Lawson or satisfies the separation axiom $\vv{T}_{2\delta}$. Is the image $\Phi(X)$ closed in $Y$?
\end{problem}

\begin{problem} Let $X$ be a complete subsemilattice of a $\w_1$-Lawson semitopological semilattice $Y$. Is $X$ closed in $Y$?
\end{problem}

\section{The closedness of the partial order in Hausdorff semitopological semilattices}\label{s:o}

Observing that the partial order $\le_X:=\{(x,y)\in X\times X:xy=x\}$ of a Haudorff topological semilattice $X$ is a closed subset of $X\times X$ we can ask the following problem, considered also in \cite{BBs}, \cite{BBG} and \cite{BBR}.

\begin{problem}\label{prob2} Let $X$ be a complete Hausdorff semitopological semilattice. Is the partial order $\le_X$ closed in $X\times X$?
\end{problem}

In this section we survey some results giving partial answers to Problem~\ref{prob2}. The following theorem is proved in \cite{BBs}.

\begin{theorem}\label{t:seq2} Let $X$ be a countably tight Hausdorff semitopological semilattice. If $X$ is ${\downarrow}\w$-complete and ${\uparrow}\w_1$-complete, then the partial order $\le_X$ of $X$ is sequentially closed in $X\times X$.
\end{theorem}

This theorem implies the following partial answer to Problem~\ref{prob2}.

\begin{corollary}\label{c:o1} Let $X$ be a Hausdorff semitopological semilattice whose square $X\times X$ is sequential. If $X$ is ${\downarrow}\w$-complete and ${\uparrow}\w_1$-complete, then the partial order $\le_X$ is closed in $X\times X$.
\end{corollary}

Another partial answer to Problem~\ref{prob2} was given in \cite{BBR}.

\begin{theorem}\label{t:o2} Let $Y$ be a semitopological semilattice satisfying the separation axiom $\vec{T}_{2\delta}$. For any complete subsemilattice $X\subseteq Y$ the partial order $\le_X$ of $X$ is closed in $Y\times Y$.
\end{theorem}

A similar result holds for complete subsemilattices of $\w$-Lawson  semitopological semilattices, see \cite{BBG}.

\begin{theorem}\label{t:03}  For any complete subsemilattice $X$ of a $\w$-Lawson semitopological semilattice $Y$, the partial order $\le_X$ of $X$ is closed in $Y\times Y$.
\end{theorem}

A topologized semilattice $X$ is called 
\begin{itemize}
\item {\em ${\uparrow}$-finite} if for any $x\in X$ the upper set ${\uparrow}x$ is finite;
\item {\em ${\downarrow}$-finite} if for any $x\in X$ the lower set ${\downarrow}x$ is finite;
\item {\em well-founded} if every nonempty subset $A\subseteq X$ contains a minimal element $a\in A$ (which means that $x\not\le a$ for any $x\in A\setminus\{a\}$); 
\item a {\em $U$-semilattice} if for any open set $U\subseteq X$ and point $x\in U$ there exists a point $u\in U$ whose upper  set ${\uparrow}u$ contains $x$ in its interior;
\item a {\em $V$-semilattice} if for any point $x\in X$ and $y\in X\setminus{\downarrow}x$ there exists a point $z\in X\setminus {\downarrow}x$ whose upper set ${\uparrow}z$ contains the point $y$ in its interior.
\end{itemize}

\begin{proposition}\label{p:well}
\begin{enumerate}
\item Each ${\downarrow}$-finite  semilattice is well-founded.
\item Each well-founded semitopological semilattice is a $U$-semilattice.
\item Each $V$-semilattice is $\downarrow$-closed;
\item A topologized $U$-semilattice is a $V$-semilattice if and only if it is $\downarrow$-closed.
\item Each semitopological $V$-semilattice is $1$-Lawson.
\end{enumerate}
\end{proposition}

\begin{proof} 1. Assume that the semilattice $X$ is $\downarrow$-finite. Given any nonempty subset $A\subseteq X$, fix any point $b\in A$ and consider the finite set ${\downarrow}b$ and its nonempty subset $A\cap{\downarrow}b$. Being finite, this set contains a minimal element $a\in A\cap {\downarrow}b$, which remains minimal in the set $A$, witnessing that $A$ is well-founded.
\smallskip

2. Assume that $X$ is a well-founded semitopological semilattice. To show that $X$ is a $U$-semilattice, fix any open set $U\subseteq X$ and point $x\in U$. Since $X$ is well-founded, the nonempty set $U\cap{\downarrow}x\ni x$ contains a minimal element $u\in U\cap{\downarrow}x$. Since $ux=u\in U$ and $X$ is a semitopological semilattice, the point $x$ has a neighborhood $O_x\subseteq X$ such that $uO_x\subseteq U$. Observe that for any $z\in O_x$ we have $uz\in U\cap{\downarrow}u\subseteq U\cap{\downarrow}x$. The minimality of the element $u$ in the set $U\cap{\downarrow}x$ ensures that $u\le uz\le z$ and hence $z\in{\uparrow}u$ and $O_x\subset{\uparrow}u$. Consequently, ${\uparrow}u$ contains $x$ in its interior, witnessing that $X$ is a $U$-semilattice.
\smallskip

3. Assume that $X$ is a $V$-semilattice. To show that $X$ is ${\downarrow}$-closed, we need to check that for any $x\in X$ the lower set ${\downarrow}x$ is closed in $X$. Since $X$ is a $V$-semilattice, for any $y\in X\setminus{\downarrow}x$ there exists an element $z\in X\setminus{\downarrow}x$ whose upper set ${\uparrow}z$ contains some neighborhood $O_y$ of $y$ in $X$.
It follows that $y\in O_y\subset{\uparrow}z\subseteq X\setminus{\downarrow}x$, witnessing that the set ${\downarrow}x$ in closed in $X$.
\smallskip

4. Assume that $X$ is a ${\downarrow}$-closed $U$-semilattice. Then for any $x\in X$ and $y\in X\setminus{\downarrow}x$, the set $W:=X\setminus{\downarrow}x$ is an open neighborhood of $y$. Since $X$ is a $U$-semilattice, the set $W$ contains an element $z\in X\setminus{\downarrow}x$ whose upper set ${\uparrow}z$ contains the point $y$ in its interior, witnessing that $X$ is a $V$-semilattice.
\smallskip

5. Assume that $X$ is a semitopological $V$-semilattice. To prove that $X$ is $1$-Lawson, fix any distinct elements $x,y\in X$. If $x\notin{\downarrow}y$, then we can find an element $z\in X\setminus {\downarrow}y$ whose upper set ${\uparrow}z$ contains $x$ in its interior. Then the closed subsemilattice ${\uparrow}z$ is a neighborhood of $x$ that does not contain $y$. If $x\in{\downarrow}y$, then $y\notin{\downarrow}x$ and we can find a point $v\in X\setminus{\downarrow}x$ whose upper set ${\uparrow}v$ contains $y$ in its interior. Since $X$ is a semitopological semilattice, the closure $\overline{X\setminus{\uparrow}v}$ of the subsemilattice $X\setminus{\uparrow}v$ is a closed subsemilattice in $X$, which is a neighborhood of $x$ that does not contain $y$. In both cases we have found a closed subsemilattice of $X$ which is a neighborhood of $x$ that does not contain $y$. Therefore, the semitopological semilattice $X$ is $1$-Lawson.
\end{proof}

\begin{proposition}\label{p:o4} For any $\uparrow$-closed topologized $V$-semilattice $X$, the partial order $\le_X$ of $X$ is closed in $X\times X$.
\end{proposition} 

\begin{proof} Given any pair $(x,y)\notin\; \le_X$, we conclude that $x\not\le y$ and hence $x\notin{\downarrow}y$. Since $X$ is a $V$-semilattice, there exists a point $z\in X\setminus{\downarrow}y$ such that ${\uparrow}z$ contains $x$ in its interior $({\uparrow}z)^\circ$. 
Then $U_x:=({\uparrow}z)^\circ$ and $U_y=X\setminus{\uparrow}z$ are two open sets in $X$ such that $U_x\times U_y$ is disjoint with $P$. So $\le_X$ is closed in $X\times X$.
\end{proof}

Propositions~\ref{p:o4} and \ref{p:well} imply

\begin{corollary}\label{c:U} For a semitopological $U$-semilattice $X$ the following conditions are equivalent:
\begin{enumerate}
\item the partial order $\le_X$ of $X$ is closed in $X\times X$;
\item $X$ is Hausdorff;
\item  $X$ is ${\downarrow}{\uparrow}$-closed.
\end{enumerate}
\end{corollary}

\begin{proof} The implication $(1)\Ra(2)$ is well-known and its proof can be found in \cite[VI-1.4]{Bible}.
\smallskip

$(2)\Ra(3)$ If $X$ is Hausdorff, then for any $x\in X$ the sets ${\downarrow}x=\{y\in X:xy=y\}$ and ${\uparrow}x=\{y\in X:xy=x\}$ are closed by the continuity of the shift $s_x:X\to X$, $s_x:y\mapsto xy$.
\smallskip

$(3)\Ra(1)$ Assume that $X$ is ${\downarrow}{\uparrow}$-closed. By Proposition~\ref{p:well}(4), $X$ is a $V$-semilattice and by Proposition~\ref{p:o4}, the partial order $\le_X$ in closed in $X\times X$.
\end{proof}

Corollary~\ref{c:U} and Proposition~\ref{p:well} imply the following corollaries.

\begin{corollary} For any well-founded semitopological semilattice $X$ the following conditions are equivalent:
\begin{enumerate}
\item the partial order $\le_X$ of $X$ is closed in $X\times X$;
\item $X$ is Hausdorff;
\item  $X$ is ${\downarrow}{\uparrow}$-closed.
\end{enumerate}
\end{corollary}

\begin{corollary}\label{c:down-finite} For any $\downarrow$-finite semitopological semilattice $X$ the following conditions are equivalent:
\begin{enumerate}
\item the partial order $\le_X$ of $X$ is closed in $X\times X$;
\item $X$ is Hausdorff;
\item  $X$ is ${\downarrow}{\uparrow}$-closed.
\end{enumerate}
\end{corollary}

Corollary~\ref{c:o1}, Theorems~\ref{t:o2}, \ref{t:03} and Proposition~\ref{p:o4} imply the following partial answer to Problem~\ref{prob2}.

\begin{corollary}\label{c:final2} For a complete semitopological semilattice $X$, its partial order $\le_X$ is closed in $X\times X$ if one of the following conditions is satisfied:
\begin{enumerate}
\item $X$ is a Hausdorff topological semilattice;
\item the topologized semilattice $X$ is $\w$-Lawson;
\item the topological space $X\times X$ is sequential and Hausdorff;
\item the topological space $X$ satisfies the separation axiom $\vv{T}_{2\delta}$;
\item the topological space $X$ is functionally Hausdorff.
\end{enumerate}
\end{corollary}

The following  examples constructed in \cite{BBR2} and \cite{BBR3} show that the completeness of $X$ in Corollary~\ref{c:final2} is essential, and Corollary~\ref{c:down-finite} has no counterparts for $\uparrow$-finite semitopological semilattices.

\begin{example}[\cite{BBR2}] There exists an ${\uparrow}$-finite metrizable semitopological semilattice $X$ whose partial order $\le_X$ is a dense non-closed subset of $X\times X$.
\end{example}

\begin{example}[\cite{BBR3}] There exists a ${\uparrow}$-finite metrizable Lawson semitopological semilattice $X$ whose partial order $\le_X$ is not closed in $X\times X$.
\end{example}

Corollary~\ref{c:final2}(2) motivates the following problem first posed in \cite{BBG}.

\begin{problem} Is the partial order $\le_X$ of an $\w_1$-Lawson complete semitopological semilattice $X$ closed in $X\times X$?
\end{problem}

%\newpage


\begin{thebibliography}{}
\bibitem{Ban} T.~Banakh, {\em Categorically closed topological groups}, Axioms 6(3)  (2017), 23.% (https://doi.org/10.3390/axioms6030023).

\bibitem{BBm} T.~Banakh, S.~Bardyla,
\emph{Characterizing chain-finite and chain-compact topological semilattices}, Semigroup Forum {\bf 98}:2 (2019), 234--250.% (https://doi.org/10.1007/s00233-018-9921-x).

\bibitem{BBc} T.~Banakh, S.~Bardyla, {\em
Completeness and absolute $H$-closedness of topological
semilattices}, Topology Appl.  {\bf 260} (2019) 189--202. 

\bibitem{BBs} T.~Banakh, S.~Bardyla, {\em On images of complete subsemilattices in sequential  semitopological  semilattices}, Semigroup Forum. {\bf 100} (2020) 662--670. %(doi.org/10.1007/s00233-019-10061-w).

\bibitem{BBw} T.~Banakh, S.~Bardyla, {\em The interplay between weak topologies on topological semilattices}, Topology Appl. {\bf 259} (2019), 134--154.% (https://doi.org/10.1016/j.topol.2019.02.028).

\bibitem{BBG} T.~Banakh, S.~Bardyla, O.~Gutik, {\em The Lawson number of a  semitopological semilattice}, Semigroup Forum. {\bf 103} (2021) 24--37. %(arXiv:1910.00436).

\bibitem{BBR} T.~Banakh, S.~Bardyla, A.~Ravsky, {\em  The closedness of complete subsemilattices in functionally Hausdorff semitopological semilattices}, Topology Appl. {\bf 267} (2019) 106874; (doi.org/10.1016/j.topol.2019.106874).

\bibitem{BBR2} T.~Banakh, S.~Bardyla, A.~Ravsky, {\em A metrizable  semitopological semilattice with non-closed partial order}, Top. Algebra Appl. {\bf 8}:1 (2020) 67--75. %preprint (arXiv:1902.08760).

\bibitem{BBR3} T.~Banakh, S.~Bardyla, A.~Ravsky, {\em A metrizable Lawson  semitopological semilattice with non-closed partial order}, Proc. Intern. Geom. Center {\bf 13}:3 (2020) 10--17. %preprint (arXiv:1909.04298).

%\bibitem{BanRav2001} T.~Banakh, A.~Ravsky, \emph{On $H$-closed paratopological groups}, The Ukrainian Congress of Mathematics, Kyiv, 2001.

%\bibitem{BB2}
%T.~Banakh, S.~Bardyla,
%\emph{Detecting (absolutely) $\mathcal{H}$-closed discrete
%semigroups},
%preprint, 2017.

\bibitem{B-G-2012}
S.~Bardyla, O.~Gutik,
\emph{On $H$-complete topological semilattices},
Mat. Stud. {\bf 38}:2 (2012), 118--123.

%\bibitem{GutikRepovs2008}
%O.~Gutik, D.~Repov\v{s}, \textit{On linearly ordered $H$-closed
%topological semilattices}, Semigroup Forum \textbf{77}:3 (2008), 474--481

%\bibitem{BanRav2001}
%T.~Banakh, A.~Ravsky,
%\emph{On $H$-closed paratopological groups},
%The Ukrainian Congress of Mathematics, Kyiv, 2001.

%\bibitem{BardGut-2016(2)}
%S.~Bardyla, O.~Gutik,
%\emph{On a complete topological inverse polycyclic monoid},
%Carpathian Math. Publ. {\bf 8}:2 (2016), 183--194.

\bibitem{B-G-R}
S.~Bardyla, O.~Gutik, A.~Ravsky,
\emph{$H$-closed quasitopological groups},
Topology Appl. \textbf{217} (2017), 51--58.


%\bibitem{Batikova2009}
%B.~Bat\'ikov\'a,
%\emph{Completion of quasi-topological groups},
%Topology Appl. \textbf{156} (2009), 2123--2128.

\bibitem{CHK}
J.H.~Carruth, J.A.~Hildebrant,  R.J.~Koch, \emph{The Theory of
Topological Semigroups},
%Vol. I, Marcel Dekker, Inc., New York and Basel, 1983;
Vol. II, Marcel Dekker, Inc., New York and Basel, 1986.

%\bibitem{CHR} W.W.~Comfort, K.H.~Hofmann, D.~Remus, {\em Topological groups and semigroups}, in: Recent progress in general topology (Prague, 1991),  North-Holland, Amsterdam, (1992) 57--144.

%\bibitem{vD} E.K.~van Douwen, {\em The integers and topology}, in: Handbook of set-theoretic topology, North-Holland, Amsterdam, (1984) 111--167,

%

\bibitem{Bruns} G.~Bruns, {\em A lemma on directed sets and chains}, Arch. der Math. {\bf 18} (1967), 561--563.

\bibitem{D-U-1998} D.~Dikranjan, V.V.~Uspenskij, {\em Categorically compact topological groups}, J. Pure Appl. Algebra {\bf 126} (1998), 149--168.

\bibitem{Iwamura}  T. Iwamura, {\em A lemma on directed sets}, Zenkoku Shijo Sugaku Danwakai {\bf 262} (1944), 107--111 (in Japanese).

\bibitem{Engelking1989}
R.~Engelking, \emph{General Topology}, Heldermann,
Berlin, 1989.


%\bibitem{DTon} D.~Dikranjan, A.~Tonolo, {\em On a characterization of linear compactness}, Riv.~Mat.~Pura Appl. {\bf 16} (1995) 95--106.

%\bibitem{DU} D.~Dikranjan, V.V.~Uspenskij, {\em Categorically compact topological groups}, J. Pure Appl. Algebra, {\bf 126} (1998), 149--168.



\bibitem{Bible}
G.~Gierz, K.H.~Hofmann, K.~Keimel, J.D.~Lawson,
M.W.~Mislove, D.S.~Scott, \emph{Continuous Lattices and Domains}.
Cambridge Univ. Press, Cambridge, 2003.

%\bibitem{Gutik-2017}
%O.V.~Gutik,
%\emph{Topological properties of the Taimanov semigroup},
%Math. Bull. of Shevchenko Sci. Soc. {\bf 13} (2016), 29--34.

\bibitem{Gutik-2014}
O.V.~Gutik,
\emph{On closures in semitopological inverse semigroups with
continuous inversion},
Algebra Discrete Math. \textbf{18}:1 (2014), 59--85.

%\bibitem{Gutik-Pavlyk-2003}
%O.V.~Gutik, K.P.~Pavlyk,
%{\em Topological Brandt $\lambda$-extensions of absolutely
%$H$-closed topological inverse semigroups},
%Visnyk Lviv Univ., Ser. Mekh.-Math. \textbf{61} (2003), 98--105.

\bibitem{G-P-R-2010}
O.~Gutik, D.~Pagon, D.~Repov\v{s}, \emph{On chains in H-closed
topological pospaces}, Order \textbf{27}:1 (2010), 69--81.

\bibitem{G-P-2001}
O.~Gutik, K.~Pavlyk, {\it $H$-closed topological
semigroups and Brandt $\lambda$-extensions}, Math. Methods and
Phys.-Mech. Fields \textbf{44}:3 (2001), 20--28, (in Ukrainian).


%\bibitem{GutikPavlyk2003}
%O.~Gutik, K.~Pavlyk, {\it Topological Brandt $\lambda$-extensions
%%of absolutely $H$-closed topological inverse semigroups}, Visnyk
%Lviv. Univ. Ser. Mekh.-Mat. {\bf 61} (2003), 98--105.

\bibitem{G-R-2008}
O.~Gutik, D.~Repov\v{s}, {\em On linearly ordered $H$-closed
topological semilattices}, Semigroup Forum \textbf{77}:3 (2008),
474--481.

%\bibitem{HS} N.~Hindman, D.~Strauss, {\em Algebra in the
%Stone-\v Cech compactifiation}, Walter de Gruyter, Berlin, NY,
%1998.

%\bibitem{Law69} J.D.~Lawson, {\em Topological semilattices with small semilattices}, J. London Math., \textbf{2} (1969), 719--724.

%\bibitem{Law74} J.D.~Lawson, {\em Joint continuity in semitopological semigroups}, Illinois J. Math. \textbf{18}:2 (1974), 275--285.


%\bibitem{MS} M.~Malliaris, S.~Shelah, {\em General topology meets model theory, on $\mathfrak p$ and $\mathfrak t$}, Proc. Natl. Acad. Sci. USA {\bf 110}:33 (2013), 13300--13305.

\bibitem{Mark} G.~Markowsky, {\em Chain-complete posets and directed sets with applications}, Algebra Universalis, {\bf 6} (1976) 53-68.

%\bibitem{Luk} A.~Luk\`{a}cs, {\em Compact-like topological groups}, Heldermann Verlag, 2009.


\bibitem{Raikov1946} D.A.~Raikov, \emph{On a completion of topological groups}, Izv. Akad. Nauk SSSR \textbf{10}:6 (1946), 513--528 (in Russian).

\bibitem{R-2003}
A.V. Ravsky,
\emph{On $H$-closed paratopological groups},
Visnyk Lviv Univ., Ser. Mekh.-Mat. \textbf{61} (2003), 172--179.

\bibitem{Stepp1969}
J.W.~Stepp, {\it A note on maximal locally compact semigroups}.
Proc. Amer. Math. Soc. {\bf 20}  (1969), 251---253.

\bibitem{Stepp1975}
J.W.~Stepp, {\it Algebraic maximal semilattices}. Pacific J. Math.
{\bf 58}:1  (1975), 243---248.

%\bibitem{Vau} J.~Vaughan, {\em Small uncountable cardinals and topology}, in: Open Problems in Topology,  North-Holland, Amsterdam, (1990) 195--218.




\end{thebibliography}
\end{document}